\documentclass[11pt]{amsart}
\usepackage{amsmath,amsfonts,latexsym,amssymb,amscd, graphicx,color}

\newcommand{\concat}{\ensuremath{+\!\!\!\!+\,}}
\renewcommand{\P}{{\mathbb P}}
\newcommand{\E}{{\mathbb E}}
\newcommand{\R}{{\mathbb R}}

\newcommand{\ZZ}{{\mathbb Z}}
\newcommand{\NN}{{\mathbb N}}
\newcommand{\GG}{{\mathbb G}}
\newcommand{\FF}{{\mathbb F}}

\newcommand{\0}{{\mathbf 0}}

\newcommand{\cc}{{\mathbf c}}
\newcommand{\zz}{{\mathbf z}}

\newcommand{\yy}{{\mathbf y}}

\newcommand{\xx}{{\mathbf x}}

\newcommand{\ee}{{\mathbf e}}
\newcommand{\dd}{{\mathbf d}}

\newcommand{\cal}{\mathcal}

\newcommand{\F}{{\cal F}}

\newcommand{\Exp}{{\rm Exp}}

\newcommand{\argmax}{\mathop{\rm arg\,max}}

\newtheorem{thm}{Theorem}
\newtheorem{coro}{Corollary}

\newtheorem{lem}{Lemma}
\newtheorem{prop}{Proposition}
\newtheorem{rem}{Remark}

\numberwithin{equation}{section}

\begin{document}
	\title[Geodesic Forests in LPP]{Geodesic Forests in the Last-Passage Percolation}
	
	\author[Sergio I. L\'opez]{Sergio I. L\'opez*}
	\author[Leandro P. R. Pimentel]{Leandro P. R. Pimentel$\dag$}
	
	\address[*]{Departamento  de Matem\'aticas\\ Facultad de Ciencias, UNAM\\
		C.P. 04510, Distrito Federal, M\'exico}
	\address[$\dag$]{Instituto de Matem\'atica\\ Universidade Federal do Rio de Janeiro\\
		Caixa Postal 68530, CEP 21941-909 Rio de Janeiro, RJ, Brasil}
	
	\email[*]{silo@ciencias.unam.mx}
	\email[$\dag$]{leandro@im.ufrj.br, lprpimentel@gmail.com}

	\thanks{Leandro P. R. Pimentel was partially supported by the CNPQ grant 474233/2012-0. Sergio I. L\'opez was partially funded by the CAPES
		grant \textit{Jovens Talentos} 063/2013}
	
	\date{\today}

	\begin{abstract}
		The aim of this article is to study the forest composed by point-to-line  geodesics in the last-passage percolation model with exponential weights. We will show that the location of the root can be described in terms of the maxima of a random walk, whose distribution will depend on the geometry of the substrate (line). For flat substrates, we will get power law behaviour of the height function, study its scaling limit, and describe it in terms of variational problems involving the Airy process.       
	\end{abstract}
	
	\maketitle
	
	\section{Introduction}

	\subsection{Introduction}
	The motivation of this article comes from the work of T. Antunovi\'c and E. B. Procaccia  \cite{AP} on \emph{geodesic forests} in first-passage percolation models (we restrict ourselves to the square lattice context). Give a bi-infinite nearest-neighbour path $\phi$ (also called line or substrate), the geodesic forest is the collection of paths composed by point-to-$\phi$ geodesics. It was proven by them that, if the initial substrate is flat, then a.s. every geodesic tree in the forest is finite. On the other hand, it is expected that if the initial substrate $\phi$ has a macroscopic convex wedge, then the tree rooted at the origin is infinite (percolation phenomena)\footnote{This problem was communicated to us by D. Ahlberg, as a conjecture proposed by I. Benjamini.}, with positive probability. We address the reader to \cite{DH} for further discussions on the first-passage percolation model with exponential passage times (Richardson  model). 
	
	The results proved by Antunovi\'c and Procaccia \cite{AP} in the first-passage percolation context can be extended mutatis mutandis to last-passage percolation models, under fairly general assumptions on the weight distribution. On the other hand, the percolation phenomena is expected to occur for substrates with a macroscopic concave wedge (we call it the concave wedge conjecture). General last-passage or first-passage percolation models are known to be very hard to analise, and still fundamental questions concerning the shape function have not yet been solved. These difficulties impose serious obstacles to understand the geometry of geodesics.   
	
	However, there are a few exceptions where the shape function is explicitly known and fluctuations results also are available. In this article we will consider one of them, namely, the exponential last-passage percolation model, where the weights are sampled from the exponential distribution. It is well known that the exponential model enjoys some crucial symmetries (like Burke's property) that allows one to find nice formulas for related invariant measures, and to use them to compute important objects, such as the shape function \cite{Ro} and the probability distribution of the asymptotic slope of the competition interface \cite{CaPi2,FePi}. Based on these special properties, we will give a positive answer to the concave wedge conjecture, and show that the probability of percolation, in a fixed direction, of the tree rooted at the origin equals the probability that a two-sided random walk with a negative drift stays below $0$. This will follow from a distributional description of the the location of the root in terms of the location of the maxima of such random walk, and it will also allow us to study the number of disjoint trees that percolates.
	
	For flat substrates, we will prove a power law behaviour of the height of a tree, with exponent $2/3$, and get some results that partially describes the limit, in the $m^{3/2}$ scale, of the maxima of the height function on a interval of size $m$. These results will connect this scenario with variational problems involving the Airy process and, consequently, with the Kardar-Parisi-Zhang (KPZ) universality class. We will also relate the height of a tree with coalescence times of semi-infinite geodesics.
	
	The proofs of the aforementioned results are not technically demanding, and they rely on the relation between geodesics and the associated Busemann field \cite{CaPi}. They parallel the method developed in \cite{CaPi2} to obtain the asymptotic slope of completion interfaces. For flat substrates, the proofs of the power law and scaling results make use of scaling properties of a point process composed by locations of maxima.  A similar approach can be found in \cite{Pi}, to deal with coalescence times of semi-infinite geodesics.
	
	\section{Definitions and Results} 
	
	\subsection{Exponential Last-Passage Percolation}
	Consider a collection of i.i.d. random variables $\{W_\xx\,:\,\xx\in\ZZ^2\}$ (also called weights), distributed according to an exponential distribution function of parameter one. In last-passage site percolation (LPP) models, each number $W_\xx$ is interpreted as the passage (or percolation) time through  vertex $\xx=(x(1),x(2))$. For $\ZZ^2$ lattice vertices $\xx\leq \yy$ (i.e. $x(i)\leq y(i)\,,i=1,2$), denote $\Gamma(\xx,\yy)$ the set of all up-right oriented paths $\gamma=(\xx_0,\xx_1\dots,\xx_k)$ from $\xx$ to $\yy$, i.e. $\xx_0=\xx$, $\xx_k=\yy$ and $\xx_{j+1}-\xx_j\in\{\ee_1,\ee_2\}$, for $j=0,\dots,k-1$, where $\ee_1:=(1,0)$ and $\ee_2=(0,1)$. The weight (or passage time) along $\gamma$ is defined as 
	$$W(\gamma):=\sum_{j=0}^{k} W_{\xx_i}\,.$$ 
	The \emph{last-passage time} between $\xx$ and $\yy$ (point to point) is defined as 
	\begin{equation*}
		L(\xx,\yy):=\max_{\gamma\in\Gamma(\xx,\yy)}W(\gamma)\,.
	\end{equation*}
	The \emph{geodesic} from $\xx$ to $\yy$ is the a.s. unique maximising path $\gamma(\xx,\yy)\in\Gamma(\xx,\yy)$ such that   
	$$L(\xx,\yy)=W(\gamma(\xx,\yy))\,.$$
	
	Let $\phi=(\phi_z)_{z\in\ZZ}$ denote a down-right bi-infinite path in $\ZZ^2$  passing through the origin: $\phi_0=\0$ and $\phi_{z+1}-\phi_z\in\{-\ee_2,\ee_1\}$. The substrate can be deterministic or random, and in the last case we are always considering
	it being independent from the weights $\{W_\xx\,:\,\xx\in\ZZ^2\}$. Let us denote by $\P_\phi$ the law of the geodesic forest for a given substrate $\phi$ and by $\P$ the law of the geodesic
	forest where $\phi$ is also random. The path $\phi$ splits $\ZZ^2$ into two disjoint regions, and we take 
	$$\Upsilon:=\left\{\xx\in\ZZ^2\,:\,\exists\,z\in\ZZ\mbox{ s.t. }\phi_z<\xx\right\}\,.$$
	We call $\phi$ the initial substrate and $\Upsilon$ the growth region. We assume that $\phi$ has a macroscopic concave wedge, i.e. there exist $\lambda_-,\lambda_+\in(0,\infty)$, with $\lambda_->\lambda_+$, such that
	\begin{equation}\label{slope}
		\lim_{z\to-\infty}\frac{\phi_z(2)}{\phi_z(1)}= -\lambda_-\,\,\mbox{ and }\,\,\lim_{z\to+\infty}\frac{\phi_z(2)}{\phi_z(1)}= -\lambda_+\,,
	\end{equation}
	(where $\phi_z(i)$ is the $i$-coordinate of $\phi_z$). We say that $\zz=\phi_z$ is a (microscopic) concave corner of $\phi$ if $\phi_z-\phi_{z-1}=-\ee_2$ and $\phi_{z+1}-\phi_z=\ee_1$, and we denote $\cal{C}(\phi)$ the set of all concave corners of $\phi$. 
	 We also define 
	$$\cal{C}_\xx(\phi):=\left\{\zz\in\cal{C}(\phi)\,:\, \zz< \xx\right\}\,.$$
	Hence $\#\cal{C}_\xx(\phi)<\infty$ for all $\xx\in\Upsilon$. 
	\begin{figure}
		\includegraphics[height=6cm]{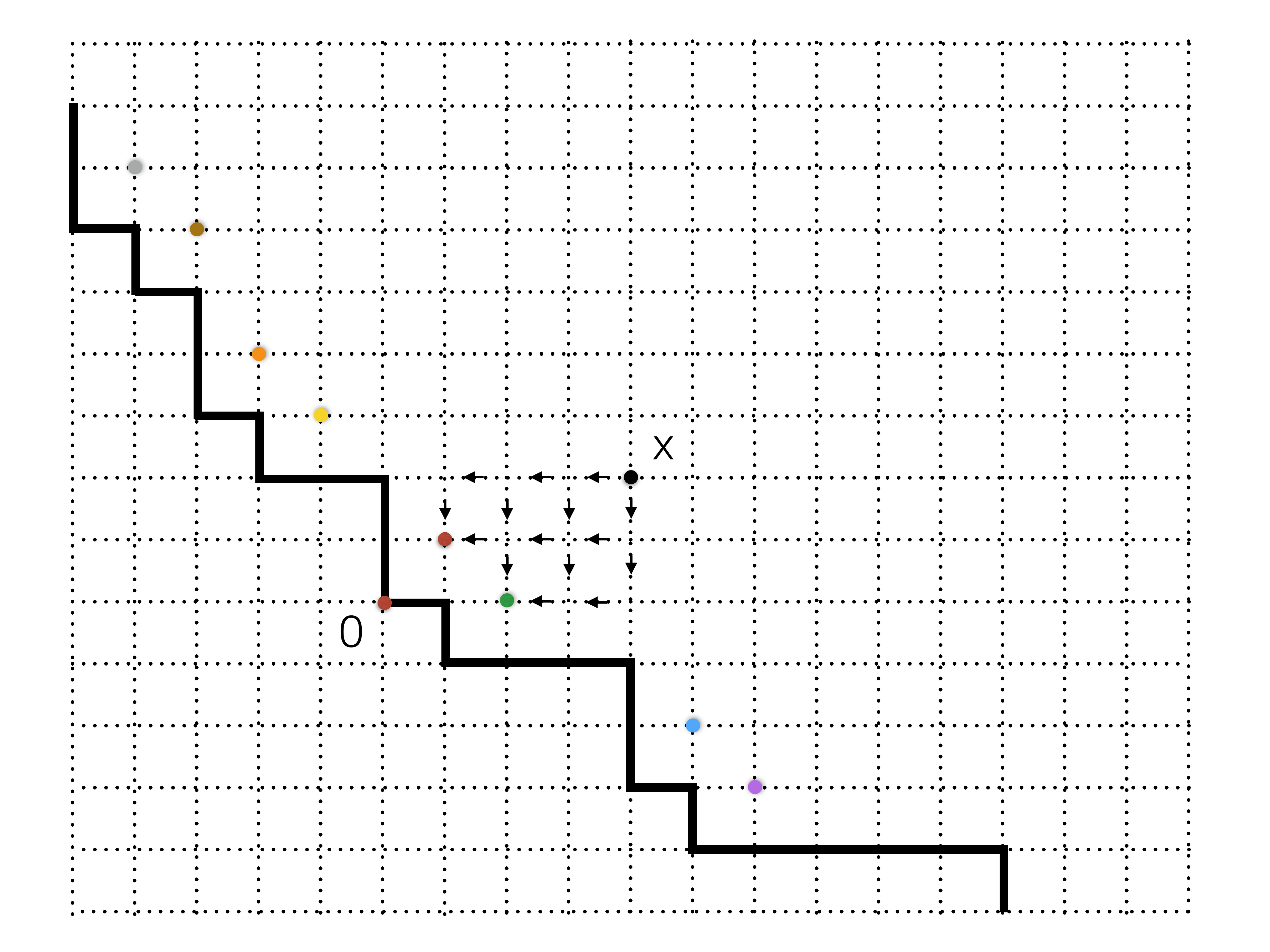}
		\caption{The substrate (black line) and the point to substrate paths.}
	\end{figure}
		
	Denote $\dd=(1,1)$ and define the (backward)  \emph{point to line last-passage time} from $\xx\in\Upsilon$ to $\phi$ as 
	$$L_\phi(\xx):=\max_{\zz\in\cal{C}_\xx(\phi)} L(\zz+\dd,\xx)\,.$$
	We note that we could have take the maximisation  over all possible $\zz$ in the substrate (there a finite number of them), however the maximum path would always start at a concave corner. The geodesic between $\phi$ and $\xx\in\Upsilon$ is defined as $\gamma_{\phi}(\xx):=\gamma(\Phi(\xx)+\dd,\xx)$, where 
	\begin{equation}\label{rootmax}
		\Phi(\xx):=\argmax _{\zz\in\cal{C}_\xx(\phi)} L(\zz+\dd,\xx)\,,
	\end{equation}
	so that, 
	$$L_\phi(\xx)=W(\gamma_{\phi}(\xx))\,.$$
	We call $\Phi(\xx)$ the root of $\gamma_{\phi}(\xx)$. If $\Phi(\xx)=\zz$ we also say that $\xx$ has root $\zz$. 
	
	We define the \emph{geodesic forest} $\F_{\phi}$ (with substrate $\phi$) as
	$$\F_{\phi}:=\left\{\gamma_\phi(\xx)\,:\,\xx\in \Upsilon\right\}\,.$$
	Then $\F_\phi$ is the union of \emph{geodesic trees} rooted at the concave corners of $\phi$,
	$$\F_\phi=\cup_{\zz\in\cal{C}(\phi)}\cal{T}_\zz\,,$$
	where 
	$$\cal{T}_\zz:=\left\{\gamma_\phi(\xx)\,:\,\Phi(\xx)=\zz\right\}\,.$$
	\begin{figure}
		\includegraphics[height=6cm]{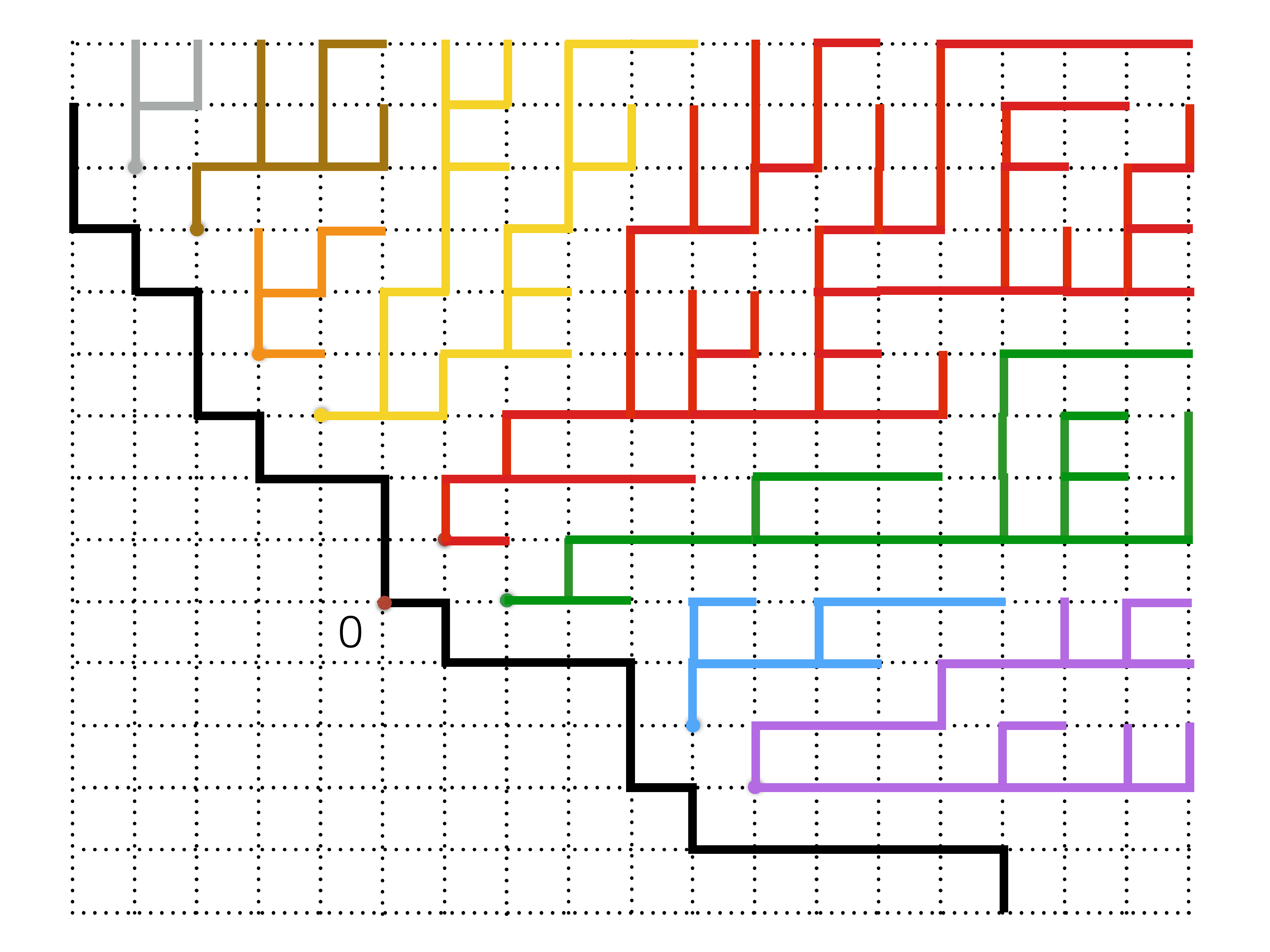}
		\caption{The geodesic forest. Each tree is coloured according to its root.}
	\end{figure}
	
	For a fixed $a>0$, we say $\zz\in\cal{C}(\phi)$ is the \emph{asymptotic root} in the direction $(1,a)$, if for every sequence of lattice points $\xx_n=(x_n(1),x_n(2))$, $n\geq 1$, in $\Upsilon$ such that 
	$$\lim_{n\to\infty}(x_n(1),x_n(2))=\infty\,\,\mbox{ and }\,\,\lim_{n\to\infty}\frac{x_n(2)}{x_n(1)}= a\,,$$
	there exists $n_0$ such that $\Phi(\xx_n)=\zz$ for all $n\geq n_0$. In this case, we denote $\Phi(a):=\zz$. The first goal of this paper is to characterise the set of directions for which there is a.s. an asymptotic root and, furthermore, to describe the distribution of the location of the root along the substrate.  
	
	To state the results we need to introduce a two-sided random walk whose distribution will depend on the initial substrate $\phi$ and on the slope $a>0$ of interest. This random walk is constructed by summing independent exponential increments along the initial substrate $\phi$. The parameter associated to the exponentials  depends on the orientation of the edge $(\phi_{z-1},\phi_{z})$ as follows. Let $\{\Exp_z(\rho):z\in\ZZ\}$ and $\{\Exp'_z(1-\rho):z\in\ZZ\}$ be independent collections of i.i.d. exponential random variables of intensity $\rho$ and $1-\rho$, respectively. These collections are also assumed to be independent of $\phi$, whenever $\phi$ is random. Denote 
	$$\rho_a:=\frac{\sqrt{a}}{1+\sqrt{a}}\,.$$
	For $z\in\{1,2,\dots\}$ let 
	\begin{equation*}
		X_{z}^{a,\phi}=\left\{\begin{array}{ll}- \Exp'_z(1-\rho_a) & \mbox{ if }\,\,\, \phi_z-\phi_{z-1}=\ee_1\,\\
			\Exp_z(\rho_a) & \mbox{ if }\,\,\, \phi_z-\phi_{z-1}=-\ee_2\,\end{array}\right.
	\end{equation*}
	(the increment along the edge $(\phi_{z-1},\phi_z)$), and for $z\in\{0,-1,-2,\dots\}$ let 
	\begin{equation*}
		X_{z}^{a,\phi}=\left\{\begin{array}{ll} \Exp'_z(1-\rho_a) & \mbox{ if }\,\,\, \phi_{z-1}-\phi_{z}=-\ee_1\,\\
			-\Exp_z(\rho_a) &  \mbox{ if }\,\,\, \phi_{z-1}-\phi_z=\ee_2\,\end{array}\right.
	\end{equation*}
	(increment along the edge $(\phi_z,\phi_{z-1})$). Set $S^{\rho,\phi}(0)=0$,  
	\begin{equation}\label{rw}
		S^{a,\phi}(z)=\sum_{k=1}^{z}X_{k}^{a,\phi}\,,\,\mbox{ for }z>0\,,\mbox{ and }\,\,S^{a,\phi}(z)=\sum_{k=z}^{-1}X_{k+1}^{a,\phi}\,,\,\mbox{ for }z<0\,.
	\end{equation}
	We note that, for $a\in(\lambda^2_{+},\lambda^2_{-})$ (\emph{the rarefaction interval}), both sides of this random walk have a negative drift.

	\begin{thm}\label{rootlaw}
		Consider the geodesic forest with a concave substrate $\phi$ satisfying \eqref{slope} and fix $a\in(\lambda^2_{+},\lambda^2_{-})$. Then $\P_\phi$-a.s. it has an asymptotic root $\Phi(a)\in\cal C(\phi)$. Furthermore, if we set $Z_\phi(a)=z\in\ZZ$ such that $\Phi(a)=\phi_z$, then
		\begin{equation}\label{substrate}
			Z_\phi(a) \stackrel{d}{=} \arg\max_{z\in\ZZ} S^{a,\phi}(z)\,.
		\end{equation}
	\end{thm}
	
	From now on we assume that the origin is a concave corner of $\phi$ and we parametrize the two sided random walk in terms of the microscopic concave corners of $\phi$. Let
	$$\dots<z_{-2}<z_{-1}<0=z_0<z_1<z_2<\dots$$ 
	denote the ordering of the concave corners of $\phi$. For $n\geq 1$ let 
	\begin{equation}\label{step}
	X^{a,+}_n:=S^{a,\phi}(z_n)-S^{a,\phi}(z_{n-1})\,\,\mbox{ and }\,\,X^{a,-}_n:=S^{a,\phi}(z_{-n})-S^{a,\phi}(z_{-n+1}) \,.
	\end{equation} 
	We set $S^{a,\pm}_0:=0$, $S^{a,\pm}_n:=\sum_{k=1}^nX^{a,\pm}_n$ for $n\geq 1$, and finally $M_a^{\pm}:=\max_{n\geq 0} S_n^{a,\pm}$. Notice that, by the definition of $X_{z}^{a,\phi}$,  concave corners are local maxima of $S^{a,\phi}$. In this way, $S^{a,\pm}$ is the sum of the increments between local maxima (it can be seen as a upper poligonal envelope of the original random walk). Thus,
	\begin{equation}\label{globalmax}
	M_a:=\max_{z\in\ZZ} S^{a,\phi}(z)=\max\left\{M_a^{+}\,,\,M_a^{-}\right\}\,,
	\end{equation}
	and, if we set $Z_a:=k\in\ZZ$ so that $Z_\phi(a)=z_k$, then 
	\begin{equation}\label{globalargmax}
	Z_a= \arg\max_{k\in\ZZ} S_k^{a,\pm}\,.
	\end{equation}
	By  independence between the sides of the random walk, for such a fixed substrate $\phi$,
	\begin{eqnarray}
		\nonumber\P_\phi \left(\Phi(a)=\0\right)&=&\P_\phi\left(Z_\phi(a)=0\right)\\
		\nonumber&=&\P_\phi\left(Z_a=0\right)\\
		\label{percolates}&=&\P_\phi\left(M^+_a=0\right)\P_\phi\left(M^-_a=0\right)\,.
	\end{eqnarray}
	
	The following corollary provides a positive answer to the concave wedge conjecture.
	\begin{coro}\label{Ben}
		Fix a substrate $\phi$ satisfying \eqref{slope}, and such that $\0$ is a microscopic concave corner. Then 
		$$\P_\phi\left(\#\cal{T}_\0=\infty \right)>0\,.$$
	\end{coro}
	
	\begin{proof}
		Since $X_1^{a,\pm}$ is the difference between two independent exponential random variables, 
		$$\P_\phi\left(X_1^{a,\pm}<-y \right)>0\,\mbox{ for all }\,y>0\,.$$
		If we denote by $\bar{M}^{\pm}_a$ the maxima of the random walk with increments $\bar{X}_n^{a,\pm}:=X_{n+1}^{a,\pm}$, then 
		$$\P_\phi\left(\bar{M}^{\pm}_a<\infty\right)=1\,$$ for 
		for $a\in(\lambda_+^2,\lambda_-^2)$, and hence
		$$\P_\phi\left(M^{\pm}_a=0\right)\geq \P_\phi\left(X^{a,\pm}_1+\bar{M}^{\pm}_a<0\right)=\int_0^\infty\P_\phi\left(X_1^{a,\pm}<-y \right)\FF_{\bar{M}^{\pm}_a}(dy)>0\,.$$
		By \eqref{percolates},
		$$\P_\phi\left(\#\cal{T}_\0=\infty \right)\geq \P_\phi\left(\Phi(a)=\0\right)=\P_\phi\left(M^+_a=0\right)\P\left(M^-_a=0\right)>0\,,$$
		which finishes the proof.
		
	\end{proof}
	
	\subsubsection{Examples of computable models} Now we proceed with some explicit calculations for some types of substrates where the distribution of the maxima can be computed. Recall now the random variables $M_a$ and $Z_a$ given by \eqref{step}, \eqref{globalmax}, \eqref{globalargmax}, and define the variables 
	$ Z^+_a = \argmax_{ k\geq 0 } S_k^{a,+}$ , $Z^-_a = -\argmax_{k \leq 0}  S_k^{a,-}$. 
	 Assume for the moment that there exists $c>0$  such that $\E (e^{ u \, X_1^{a,\pm}} )<\infty$ for all $u \in [0,c]$, and that there exists a (minimal) $\gamma^\pm >0$ such that $\E (e^{\gamma \, X_1^{a,\pm}} ) =1$. Also, assume that $X_1^{a, \pm}$ has an exponential right tail: there exist constants $b^\pm, \delta^\pm >0$ such that its density at the right of the origin has the form
	 \begin{equation}\label{exptail}
	 	f_{X_1^{a, \pm}}(x) = b^\pm \, \delta^\pm \, e^{- \delta^\pm \, x} \quad \forall x > 0, 
	 \end{equation}
	 where $b^\pm= \P(X_1^{a,\pm}>0)$. By relating the random walk with waiting times in queueing theory, it is known that \cite{Res}:
	 \begin{equation}\label{maxdist}
	 \P (M^\pm_a =0)=\frac{\gamma^\pm}{\delta^\pm}\,\,\,\mbox{ and }\,\,\,\,\P (M^\pm_a > x)= \Big(1- \frac{\gamma^\pm}{\delta^\pm} \Big) e^{- \gamma^\pm \, x} \,, \,\,\forall x \geq 0\,.
	 \end{equation}
	 \newline
		
	\noindent\paragraph{\bf Bernoulli Substrate}
	Fix $p_-\in(0,1]$ and $p_+\in[0,1)$, with $p_->p_+$. Consider a random substrate where $\phi_{-1}=\ee_2$, $\phi_0=\0$ and $\phi_1=\ee_1$. For $z>1$
	$$\P\left(\phi_{z+1} - \phi_{z}=-\ee_2\right)=p_+=1-\P\left(\phi_{z+1} - \phi_{z}=\ee_1\right)\,,$$
	while for $z<1$, 
	$$\P\left(\phi_{z-1} - \phi_{z}=\ee_2\right)=p_-=1-\P\left(\phi_{z-1} - \phi_{z}=-\ee_1\right)\,.$$
	Thus, a.s.
	$$\lambda_+=\frac{p_+}{1-p_+}\,\,\,\mbox{ and }\,\,\,\lambda_-=\frac{p_-}{1-p_-}\,.$$
	Also, for $n\geq1$,
	\begin{eqnarray}\label{stepber}
	X^{+,a}_n &\stackrel{dist.}{=}&\Exp_n((1-p_+)\rho_a)-\Exp_n(p_+(1-\rho_a))\,, \\
	\textrm{ and } \qquad \qquad \qquad \qquad \qquad \quad &  &   \nonumber \\
	X^{-,a}_n   &\stackrel{dist.}{=} &\Exp_n(p_-(1-\rho_a))-\Exp_n((1-p_-)\rho_a)\,.  \nonumber
	\end{eqnarray}
	where the last distribution equalities are with respect to the joint law $\P$. To see this, notice that the number of down steps between two right steps is distributed as the number of trials until the first success (right step) of a Bernoulli random variable of parameter $(1-p_+)$. Thus, along down steps we have a geometrical sum of exponentials of parameter $\rho_a$, which gives an exponential of parameter $(1-p_+)\rho_a$. For the other cases the argument is analog. Here condition \eqref{exptail} is met, and the parameters are known for the associated one-sided storage system (it is an M/M/1 queue \cite{Res}): $\delta_a^+=(1-p_+) \rho_a $, $\delta_a^-= p_-(1-\rho_a)$, $\gamma_a^+=(1-p_+) \rho_a - p_+ (1-\rho_a)$ and $\gamma_a^-=p_-(1-\rho_a)-(1-p_-) \rho_a$. As a consequence of \eqref{maxdist},  
	$$\P\left(M^+_a=0\right)=1-\frac{p_+(1-\rho_a)}{(1-p_+)\rho_a}=1-\frac{\lambda_+}{\sqrt{a}}\,,$$
	and 
	$$\P\left(M^-_a=0\right)=1-\frac{(1-p_-)\rho_a}{p_-(1-\rho_a)}=1-\frac{\sqrt{a}}{\lambda_-}\,.$$
	Therefore, 
	$$\P\left(\Phi(a)=\0\right)=\left(1-\frac{\lambda_+}{\sqrt{a}}\right)\left(1-\frac{\sqrt{a}}{\lambda_-}\right)\,,\,\,\mbox{ for }\,\,a\in(\lambda^2_+,\lambda^2_-)\,.$$
	By maximising the last expression we find $a=\lambda_+\lambda_-$, and thus 
	$$\P\left(\#\cal{T}_\0=\infty \right)\geq\left(1-\frac{\sqrt{\lambda_+}}{\sqrt{\lambda_-}}\right)^2\,.$$
	For this type of substrate (and the following one) it is possible to obtain more explicit expression for the joint law of the maxima and its location (see \eqref{jointexpl1} below), since the density of a difference of two independent random gamma variables is known \cite{Kl}, however it turns out to be a complex formula. 
	\newline
	
	\noindent\paragraph{\bf Periodic Substrate} 
	Let $k_+,k_-\geq 1$ such that $\max\{k_+,k_1\}\geq 2$. Define $\phi$ by starting at $\phi_0=\0$ and then, for $z>0$, jumping $k_+$ steps to the right and $1$ down, repeatedly, while for $z<0$, jumping $k_-$ steps up and $1$ to the left, repeatedly. For this substrate we have $\lambda_+=k_+^{-1}$ and $\lambda_-=k_-$. The increments are given by 
	\begin{eqnarray}\label{stepper}
	X^{+,a}_n &\stackrel{dist.}{=}& \Exp(\rho_a)-\sum_{j=1}^{k^+}\Exp_j(1-\rho_a)\,,\,\,\mbox{ for }\,\,n\geq 1\,, \\
	\textrm{ and } \qquad \qquad &  &   \nonumber \\
	 X^{-,a}_n &\stackrel{dist.}{=}&  \Exp(1-\rho_a)-\sum_{j=1}^{k^-}\Exp_j(\rho_a)\,,\,\,\mbox{ for }\,\,n\leq 1\,. \nonumber
	\end{eqnarray}
	In this case, the exponential right tail assumption \eqref{exptail} is fulfilled, and we can use formula \eqref{maxdist} to compute the probability of the maxima be zero. This is related to a G/M/1 queueing system \cite{Res} and it boils down to calculate, for each $a\in(k_+^{-2},k_-^2)$, the smallest positive solution of    
	$$\alpha(1-\rho_a\alpha)^{k_+}=(1-\rho_a)^{k_+}\,\,\mbox{ and }\,\,\alpha(1-(1-\rho_a)\alpha)^{k_-}=\rho_a^{k_-}\,,$$
    which we denote by $\alpha^+_a,\alpha^-_a\in(0,1)$, respectively. Thus,
	$$\P_\phi\left(M^+_a=0\right)=1-\alpha^+_a\,\,\mbox{ and }\,\,\P_\phi\left(M^-_a=0\right)=1-\alpha^-_a\,,$$
	and  
	$$\P_\phi\left(\Phi(a)=\0\right)=\left(1-\alpha^+_a\right)\left(1-\alpha^-_a\right)\,,\,\,\mbox{ for }\,\,a\in(k_+^{-2},k^2_-)\,.$$
	
	 We were not able to find a closed formula for other periodic substrates because there were no results on the distribution of the maxima when the step distribution of the underlying random walk is different from exponential minus gamma.      
	\newline
		
	\noindent\paragraph{\bf Finite Rooted Substrate} 
	To compute the value of the probability that the tree at the origin percolates using \eqref{substrate} one needs to have more information on the joint probability of $S^{a,\phi}$, as a process in $a\in(\lambda_+^2,\lambda_-^2)$. This problem is related to the computation of the joint distribution of the Busemann field for different values of directions (see Section \ref{Bus}), which is still not accomplished. However, there is a particular example of substrate where the probability of percolation can be computed by means of another method, developed by Coupier \cite{Co1}, that is based on the relation between completion interfaces and second-class particles \cite{FePi}, and on the distributional description of the totally asymmetric simple exclusion speed process introduced by Amir, Omer and V\'alko \cite{AOV}. Consider a substrate $\phi$ as follows: set $\phi_0=\0$, $\phi_{-1}=\ee_2$, and for $z<-1$ set $\phi_{z}=|z+1|\ee_2-\ee_1$. Fix $m\in\NN$ and for $z\in\{1,\cdots,m\}$ we set $\phi_z=z\ee_1$ while for $z>m$ we set $\phi_z=z\ee_1-\ee_2$. In this way, we get three concave corners and it might happen that the tree at the origin is finite. The probability that this happens is 
	\begin{equation}\label{finiteroots}
		\P_\phi\left(\#\cal{T}_\0<\infty\right)=\frac{2}{m+2}\,.
	\end{equation} 
	Formula \eqref{finiteroots} is a consequence of equation (22) in \cite{Co1}, as soon as one realizes that $\#\cal{T}_\0=\infty$ means coexistence of the trees. On the other hand, the probability of coexistence was obtained in \cite{Co1} using the results in \cite{AOV,FePi}. To compute the probability of percolation in the direction $(1,a)$, by our method, we only need to compare independent gamma random variables,
	\begin{eqnarray*}
		\P_\phi\left(\Phi(a)=0\right)&=&\P_\phi\left(M^+_a=0\right)\P_\phi\left(M^-_a=0\right)\\
		&=&\P \left(\Gamma(\rho_a,1)<\Gamma(1-\rho_a,m)\right)\P\left(\Gamma(1-\rho_a,1)<\Gamma(\rho_a,1)\right)\\
		&=&\frac{\sqrt{a}}{1+\sqrt{a}}\sum_{j=0}^{m-1}\frac{(j+m-1)!}{j!}\left(\frac{1}{1+\sqrt{a}}\right)^{j+1}\,.   
	\end{eqnarray*}
	
	 \subsubsection{The joint transform of the maximum and its location}
    Not easy to apply formulas are available for the joint distribution of the global maximum of a random walk and its location, however next proposition suggest some Monte Carlo method to approximate it, and thus to approximate the law of $(Z,M)$ for the Bernoulli and periodic substrates.
    
    \begin{prop}\label{ZrelN}
    Let $S_n = \sum_{i=1}^n X_i$, $S_0=0$ denote a random walk defined on the non-negative integers, such that $\mathbb{E} (X_i) <0$, where $X_i$ has a continuous distribution. Define $M:= \max_{n \geq 0} S_n$, $Z=\argmax_{n \geq 0} S_n$, and the hitting time of $\{S_n\}_{n \geq 0}$ to the negative numbers as
     $$N:= \min\{ n \geq 1 | \sum_{i=1}^{n} X_i \leq 0 \}. $$
     Then $\P( Z=0, M=0) =\P(M=0)$ and we have the equality of densities
     $$f_{Z,M}(n,x)= \sum_{j>n}f_{N,S_n} (j,x)  \P(M=0) \qquad \forall n \geq 1, x \geq 0. $$
    \end{prop}
       
     Let us briefly describe the method.  Assume the probability of $M^\pm_a$ be equal to zero is known. Then, to approximate the right hand side of the result in Proposition \ref{ZrelN} one could simulate random walks where the hitting time to the negative real has yet not happened (using acceptance-rejection method) and this would be enough to approximate the joint distribution of $(Z^\pm_a,M^\pm_a)$. Then, by using this approximation, it is possible to simulate $(Z^+_a,M^+_a) $ and $(Z^-_a,M^-_a)$ independently. Finally, $(Z_a,M_a)$ is a deterministic function of those variables, thus one could use crude Monte Carlo to approximate its distribution. 

Using the previous proposition, and the same connection between $(M^\pm_a, Z^\pm_a)$ and queueing systems, we obtained an expression for the joint transform of the maximum and its location. 
	\begin{prop}\label{jointexpl}
	Under the assumptions on the step distribution stated above, the joint transform of $(Z_a,M_a)$ can be expressed as
	\begin{eqnarray}
	\nonumber\E(s^{Z_a} \, e^{uM_a})  &=& \Big( \frac{\gamma^+}{\delta^+} \Big) \Big( \frac{\gamma^-}{\delta^-} \Big) +\phi^+_a(s,u) 
	- \Big( 1 - \frac{\gamma^+}{\delta^+} \Big) \phi^+_a(s,u - \gamma^+) \\
	\label{jointexpl1}&+&  \phi^-_a(s,u) - \Big( 1 - \frac{\gamma^-}{\delta^-} \Big)  \phi^-_a(s,u- \gamma^-), \quad \forall s \in (0,1), \, u \in [0,c] \,,
	\end{eqnarray}
	where $\gamma^\pm, \delta^\pm$ are the corresponding constants for each one-sided random walk and the function $ \phi^\pm_a$ is given by
	\begin{equation*}
		\phi^\pm_a(s,u) = \frac{\gamma^\pm}{\delta^\pm}  \exp \Big\{  \sum_{n=1}^{ \infty} \frac{s^n}{n}  \Big( \E( e^{-u (S_n^{a,\pm})^*}) - \P( S_n^{a,\pm} \leq 0 ) \Big)  \Big\} ,
	\end{equation*}
	and we are using the (non-standard) notation $x^*= \max \{0,x\}$, for $x$ real.
\end{prop} 
  To simplify the expression  \eqref{jointexpl1}, it would be necessary to compute $\delta^\pm, \gamma^\pm$, $\E(e^{-u(S_n^{a,\pm})^*}) $ and $\P (S_n^{a,\pm} \leq 0)$ in each particular case.\\

	\subsubsection{On the number of tress that percolates}
	We note that, since geodesics cannot cross (though they may coalesce) the location of the root is monotonic with respect to the slope $a$. Thus, if $c\in[a,b]\subseteq(\lambda^2_{+},\lambda^2_{-})$ then $Z_\phi(b)\leq Z_\phi(c)\leq Z_\phi(a)$. In particular, there are only finitely many roots such that the respective tree percolates within $[a,b]$.  However, as soon as one get closer to the critical slopes $\lambda_{\pm}^2$, the number of infinite trees may explode. To see an example where this occurs, take a Bernoulli type random substrate $\phi$ as before with parameters $p_-$ and $p_+$:
	$$\E X_{n}^{a,+} = \frac{1}{(1-p_+)\rho_a} - \frac{1}{p_+(1-\rho_a)}\,\,\,\mbox{ and }\,\,\,\E X_{n}^{a,-} = \frac{1}{p_-(1-\rho_a)} - \frac{1}{(1-p_-)\rho_a}\,.$$
	At the critical slopes, $\E X_{n}^{ \lambda_{\pm}^2,\pm} = 0$ and $M_{ \lambda_{\pm}^2}^\pm=\infty$. Since $Z_\phi(a)$ is a monotonic function of $a$, we must have that $Z_\phi(a) \to \pm\infty$, as $a \to \lambda^2_\pm$. Thus, we have the following corollary \footnote{A similar reasoning can be done for a deterministic concave substrate that exhibits a periodic structure on each side, like the one with parameters $k_-,k_+$, and the analog result will hold as well.}.
	
	\begin{coro}\label{infinite}
		Consider the geodesic forest composed by a Bernoulli type random substrate $\phi$ with parameters $p_-$ and $p_+$, with $p_->p_+$. Then 
		$$\lim_{a\to \lambda^2_\pm}Z_\phi(a)\stackrel{a.s.}{=}\pm\infty\,.$$
		In particular, a.s., there will be infinitely many roots that percolates.
	\end{coro}
	
	In view of Corollary \ref{infinite} one might expect that $Z_\phi(a)$ converges to infinity according to  some speed that depends on $a$. We expect that Proposition \ref{jointexpl} may be helpful to attack this problem for computable models (Bernoulli or Periodic initial substrates). Another approach would be to fix $a=1$, set $p_-=1/2+\epsilon$, $p_+=1/2-\epsilon$ and then send $\epsilon\to 0^+$.

	\subsubsection{Last-passage percolation with weighted substrate}
	An alternative description of the model can be done by fixing the initial substrate as the horizontal axis and then putting extra weights along it. Now, the geometry of the substrate is represented by a collection of non-negative real numbers $\{\nu_k:k\in\ZZ\setminus\{0\}\}$.  Define
	\begin{equation*}
	\nu(k)=\left\{\begin{array}{ll} 0 & \mbox{ if }\,\,\,k=0 \,\\
	\sum_{i=1}^k\nu_i & \mbox{ if }\,\,\,k>0 \,\\
	-\sum_{i=k}^{-1}\nu_i & \mbox{ if }\,\,\,k<0\,\end{array}\right.
	\end{equation*}
	We will assume that this collection has an asymptotic drift: there exists $\mu_->\mu_+> 1$ such that 
	$$\lim_{k\to-\infty}\frac{\nu(k)}{k}=-\mu_-\,\,\mbox{ and }\,\,\lim_{k\to\infty}\frac{\nu(k)}{k}=\mu_+\,.$$ 
	The last-passage percolation time with weighted substrate $\nu$ is defined for $x\in\ZZ$ and $n\geq 1$ as 
	\begin{equation}\label{lppsystem}
	L_\nu(x,n):=\max_{k\leq x}\left\{\nu(k)+L_k(x,n)\right\}\,
	\end{equation}
	where $L_k(x,n):=L\left((k,1),(x,n)\right)$. The point to line geodesic is now defined as $\gamma_\nu(x,n):=\gamma\left(K_\nu+\ee_2,(x,n)\right)$ where 
	$$K_\nu(x,n):=\arg\max_{k\leq x}\left\{\nu(k)+L_k(x,n)\right\}\,.$$
	
	Let
	$$\F_{\nu}:=\left\{\gamma_\nu(x,n)\,:\,x\in\ZZ\,,\,n\geq 1\right\}\,.$$
	Then $\F_\nu$ is a union of trees rooted at maximisers:
	\begin{equation}\label{wforest}
	\F_\nu=\cup_{k\in\ZZ}\cal{T}_k\,,
	\end{equation}
	where
	$$\cal{T}_k:=\left\{\gamma_\nu(x,n)\,:\,K_\nu(x,n)=k\right\}\,.$$
	We call $\F_\nu$ the geodesic forest with weighted substrate $\nu$. As before, we also say that a slope 
	$a>0$ has root $k$ if for every sequence of lattice points $(\xx_n)_{n\geq 1}$ in $\ZZ\times\NN$ with direction $(1,a)$, there exists $n_0$ such that $K_\nu(\xx_n)=k$ for all $n\geq n_0$. In that case, we denote $K_\nu(a):=k$ (the root of the direction $(1,a)$).
	
	The weighted substrate may be deterministic or random. We always assume that it is independent of the lattice weights $W_\xx$. As an example, take collections $\{\Exp^\nu_k(1-p_+):k>0\}$ and $\{\Exp^\nu_k(1-p_-):k<0\}$ of i.i.d. exponential random variables of intensity $1-p_+$ and $1-p_-$, respectively, where $p_+,p_-\in(0,1)$. In this case
	$$\mu_-=\frac{1}{1-p_-}\,\,\mbox{ and }\,\,\mu_+=\frac{1}{1-p_+}\,.$$
	The assumption $\mu_->\mu_+$ (or $p_->p_+$) corresponds to the rarefaction regime where the characteristic slopes satisfy 
	$$ (\mu_+-1)^2=\left(\frac{p_+}{1-p_+}\right)^2<\left(\frac{p_-}{1-p_-}\right)^2=(\mu_--1)^2\,.$$
	
	Similar to the preceding case, let $\{\Exp_z(1-\rho):z\in\ZZ\}$ be a collection of i.i.d. exponential random variables of intensity $1-\rho$. This collection is also assumed to be independent of $\nu$, whenever $\nu$ is random. Define 
	\begin{equation*}
	\mu_a(k)= \left\{ \begin{array}{ll} 0 & \mbox{ if }\,\,\,k=0 \,\\
	\sum_{i=1}^k\Exp_i(1-\rho_a) & \mbox{ if }\,\,\,k>0 \,\\
	-\sum_{i=k}^{-1}\Exp_i(1-\rho_a) & \mbox{ if }\,\,\,k<0\, \textrm{,} \end{array} \right. 
	\end{equation*}
	where we kept $\rho_a = \frac{\sqrt{a}}{1+\sqrt{a}}$. Fix $a\in((\mu_+-1)^2,(\mu_--1)^2)$. Then a.s. it has an asymptotic root $K_\nu(a)$. Furthermore,
	\begin{equation}\label{weighted}
	K_\nu(a)\stackrel{dist.}{=}\arg\max_{k\in\ZZ}\left\{ \nu(k)-\mu_{a}(k)\right\}\,.
	\end{equation}
	 The proof of \eqref{weighted} follows the same method developed to prove \eqref{substrate}. For the sake of brevity, we will not include it in this article.
	\newline
	
	\noindent\paragraph{\bf Exponential Weighted Substrate}
	For the weighted substrate with exponential distribution with parameters $(1-p_-)$ and $(1-p_+)$ the calculation using maxima of random walks is analog, and we get that 
	$$\P\left(\Phi(a)=\0\right)=\left(1-\frac{\mu_+}{1+\sqrt{a}}\right)\left(1-\frac{1+\sqrt{a}}{\mu_-}\right)\,,\,\,\mbox{ for }\,\,a\in((\mu_+-1)^2,(\mu_--1)^2)\,.$$
	By maximising over  $a\in((\mu_+-1)^2,(\mu_--1)^2)$, we find $(1+\sqrt{a})^{-1}=\sqrt{\mu_+\mu_-}$, and hence
	$$\P\left(\#\cal{T}_\0=\infty \right)\geq\left(1-\frac{\sqrt{\mu_+}}{\sqrt{\mu_-}}\right)^2\,.$$
	For this model we also have that a.s. $K_\nu(a)\to\pm\infty$, as  $a \to (\mu_\pm-1)^2$.
	\newline

	\subsection{Convergence of the geodesic forest with flat substrate}
	Finite geodesics do converge when we fix one end point and send the other to infinity along a prescribed direction. For simplicity we will choose the direction $(-1,-1)$. The following is a well known result in last-passage percolation with exponential weights \cite{Co,FePi}: a.s for each $\xx\in\ZZ^2$ there is a unique semi-infinite geodesic $\gamma(\xx)$ (down-left oriented) such that if a sequence of lattice points $\xx_n=(x_n(1),x_n(2))\leq \xx$, $n\geq 1$, satisfies
	$$\lim_{n\to\infty}(x_n(1),x_n(2))=-\infty\,\,\mbox{ and }\,\,\lim_{n\to\infty}\frac{x_n(2)}{x_n(1)}= 1\,,$$
	then 
	\begin{equation}\label{semigeo}
	\lim_{n\to\infty}\gamma(\xx_n,\xx)=\gamma(\xx)\,.
	\end{equation}
	Furthermore, for every $\xx,\yy\in\ZZ^2$ there is $\cc$ such that (coalescence occurs)
	\begin{equation}\label{coal}
	\gamma(\xx)=\gamma(\cc) \concat \gamma(\cc,\xx)\,\,\mbox{ and }\,\,\gamma(\yy)=\gamma(\cc)\concat\gamma(\cc,\yy)\,,
	\end{equation}
	where $\concat$ denotes the concatenation of paths. As a consequence, the collection of such semi-infinite geodesics
	$$\cal{F}:=\left\{\gamma(\xx)\,:\,\xx\in\ZZ^2\right\}\,,$$
	is a.s. a tree. Let $\nu$ be the exponential weighted substrate of parameter $1/2$. It follows from Theorem 5.3 of \cite{CaPi} that 
	$$\cal{F}^+\stackrel{dist.}{=}\cal{T}_\nu\,,$$
	where $\cal{F}^+:=\cal{F}\cap \NN\times\ZZ$ and $\cal{T}_\nu$ is given by ($\ref{wforest}$).
	
	Furthermore, the geodesic tree $\cal{F}$ is the limiting tree of a  geodesic forest with respect to a flat substrate. Indeed, assume that $\phi^n=(\phi_z^n)_{z\in\ZZ}$ has inclination
	$$\lim_{z\to-\infty}\frac{\phi^n_z(2)}{\phi^n_z(1)}=\lim_{z\to+\infty}\frac{\phi^n_z(2)}{\phi^n_z(1)}= -1\,.$$
	The index $n$ means that now the origin of the substrate is $\phi^n_0=-n\dd$, we will let $n\to\infty$. Denote $\cal{F}_{\phi^n}$ the respective geodesic forest. Then for a fixed $\xx\in\ZZ^2$ the point to substrate geodesic $\gamma_{\phi^n}(\xx)$ will have a root $\Phi_n(\xx)$ whose distance from  $-n\dd$ is of sub-linear order. The proof follows the same method used to prove Lemma 5.2 \cite{CaPi} in the Poissonian last-passage percolation model, and can be extended to the lattice context with exponential weights as well. The key is the knowledge of the curvature of the limiting shape, which is explicitly known in both cases. This implies that the sequence of finite geodesics paths $(\gamma_{\phi^n}(\xx))_{n\geq 1}$ has asymptotic direction $(-1,-1)$ and hence, by \eqref{semigeo}, will converge to the semi-infinite geodesic path $\gamma(\xx)$. Moreover, this convergence holds simultaneously for any finite collection of sites. Therefore this geodesic forest will weakly converge to $\cal{F}$, as $n\to\infty$. If one takes a flat substrate with slope $-\lambda$, the same result holds but now the limiting object  is the geodesic tree composed by semi-infinite geodesics with direction $-(1, \frac{1}{\lambda})$.
	
	\subsection{Scaling the height of a tree}
	The proof of the scaling behaviour of point to point exponential last-passage percolation times, and its connection with the Tracy-Widom distribution, was performed by Johansson \cite{Jo}. This result was later extended to convergence to the Airy process \cite{CoFePe}. Precisely:
	\begin{equation}\label{scalelpp}
	\lim_{n\to\infty}\frac{L\left(\0,( n+\lfloor 2^{2/3}xn^{2/3}\rfloor, n-\lfloor 2^{2/3}xn^{2/3}\rfloor)\right)-4n}{2^{4/3}n^{1/3}}\stackrel{dist.}{=}\cal A(x)-x^2\,,
	\end{equation}
	where $\{\cal A(x)\,:\,x\in\R\}$ is the so called Airy process ($\lfloor x\rfloor$ denotes the integer part of $x\in\R$). The characteristic exponents $\chi=1/3$ and $\xi=2/3$ are believed to describe the fluctuations of a broad class of interface growth models, named the Kardar-Parisi-Zhang (KPZ) universality class. In this section we will explain how to use \eqref{scalelpp} to shed light on the scaling scenario of geodesic forests. 
	
	\begin{figure}[h]
		\includegraphics[height=5cm]{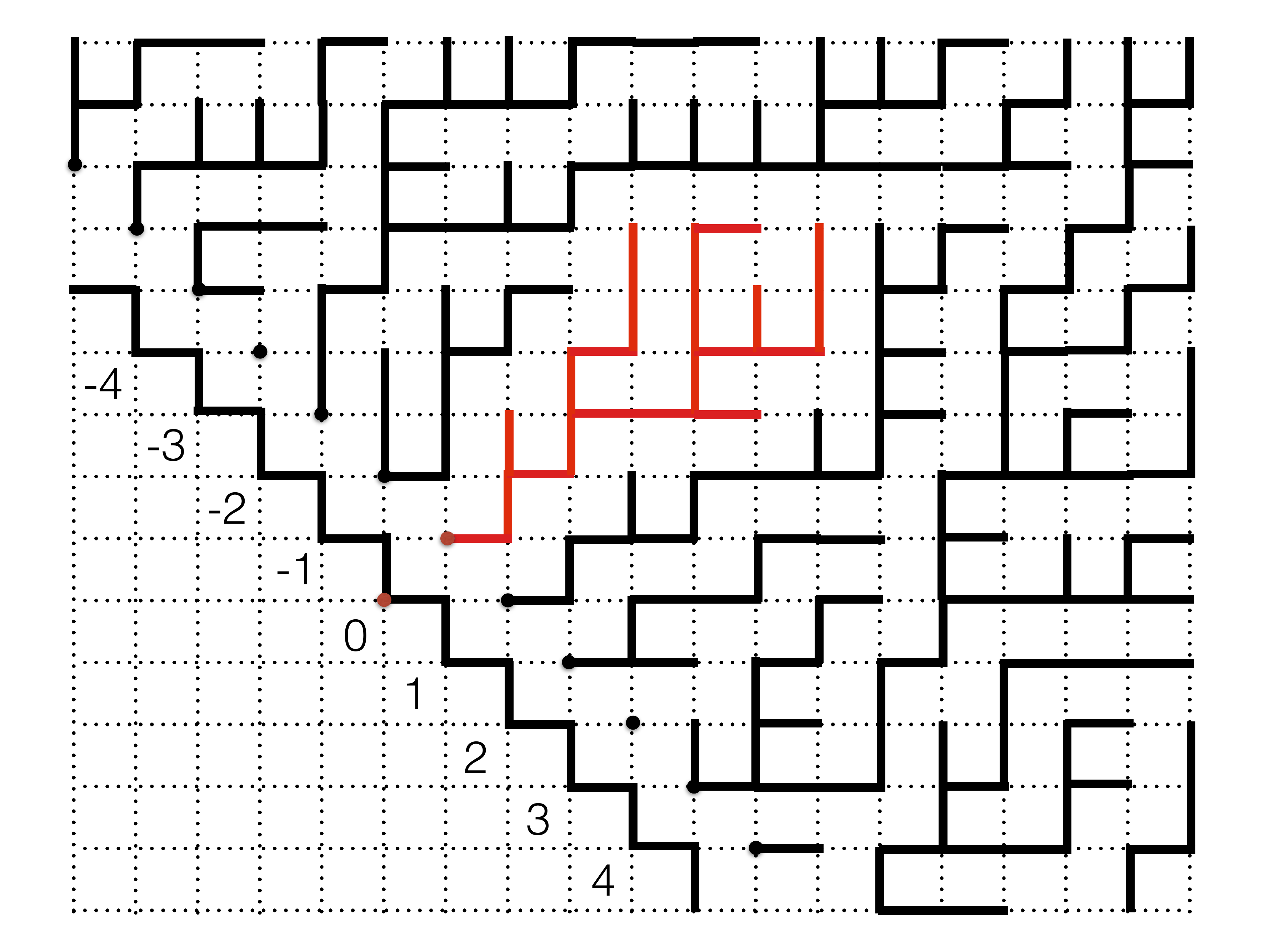}
		\caption{Flat periodic substrate. The height of the tree at $0$ is $H_0=12$.}
			\end{figure}             
	From now on we fix the substrate $\phi$ whose set of corners is given by the diagonal $\{(z,-z)\,:\,z\in\ZZ\}$. Recall that the geodesic tree with substrate $\phi$ and rooted at $(z,-z)$ was defined as the set of all point to substrate geodesics with common root at $(z,-z)$:
	$$\cal{T}_z:=\left\{\gamma_\phi(\xx)\,:\,\Phi(\xx)=(z,-z)\right\}\,.$$
	Let $L(n):=\left\{(x(1),x(2)\in\ZZ^2\,:\,x(1)+x(2)=n\right\}$. The height of  $\cal{T}_z$ is defined as 
	$$H_z:=\max\left\{n\geq 1\,:\, (x(1),x(2))\in\cal{T}_z\,\mbox{ for some }\,(x(1),x(2))\in L(n+1) \right\}\,.$$
	Then a.s. $H_z<\infty$ for all $z\in\ZZ$ (this follows from Antunovi\'c and Procaccia \cite{AP}). It is not hard to see that 
	$$\P\left(H_0\geq n\right)\geq cn^{-1}\,\,\mbox{ and }\,\,\E H_0=\infty\,.$$
	A harder task is to determine the precise tail decay of $H_0$ and the respective power law  exponent. We expect
	$$\P\left(H_0\geq  n\right)\sim n^{-2/3}\,,\mbox{ as }n\to\infty\,,$$
and the reason is related to \eqref{scalelpp}.	
	To perform a scaling limit of the height function one can take the maximum of the height among a  finite collection of trees,
	$$H(m):=\max_{z\in(0,m]}H_z\,.$$
 Again by \eqref{scalelpp}, we expect that the right scaling is $m^{3/2}$. Here we will prove the following theorems.
	\begin{thm}\label{thm:tail}
		There exists $c_0>0$ such that 
		$$\P\left(H_0\geq n\right)\geq \frac{c_0}{n^{2/3}}\,.$$
	\end{thm}
	
	\begin{thm}\label{thm:maxscal}
		Let
		$$\GG(r):=\liminf_{m\to\infty}\P\left(\frac{H(m)}{2^{-1}m^{3/2}}> r\right)\,.$$
		Then 
		\begin{equation}\label{maxscal}
		\GG(r)\geq\P\left(X \in \left(-(2r^{2/3})^{-1},(2r^{2/3})^{-1}\right] \right)\,,
		\end{equation}
		where $X:=\arg\max_{x\in\R}\left\{\cal A(x)-x^2\right\}$.
	\end{thm}
	It is known that $X$ has a density and an explicit formula can be found in \cite{MoQuRe}. Combining this with \eqref{maxscal} one get as corollary a lower bound for the limiting tail function $\GG$.   
	\begin{coro}\label{lowertail}
		Let $f$ denote the density of $X$. Then 
		$$\liminf_{r\to\infty}r^{2/3}\GG(r)\geq f(0)>0\,.$$
		Furthermore,
		$$\lim_{r\to 0^+}\GG(r)=1\,.$$
	\end{coro}
	
	\subsubsection{The weighted substrate model and coalescence times}
	Consider the exponential weighted substrate model with parameter $p_-=p_+=1/2$. The height of a tree is now defined as 
	$$H_k:=\max\left\{n\geq 0\,:\,(x,n)\in\cal{T}_k \cup\{(k,0)\}\,,\mbox{ for some }\,x\in\ZZ \right\}\,,$$
	and the maximum $H(m)$ over $(0,m]$ has an analog definition. The analogs to Theorem 3 and Theorem 4 can be obtained for this model as well, but now the rescaling factor for $H(m)$ is $2^{-5/2}m^{3/2}$ and the lower bound is giving as function of the distribution of $\bar X :=\arg\max_{x\in\R}\left\{\sqrt{2}\cal B(x)+\cal A(x)-x^2\right\}$, where $\cal B$ is a standard Brownian motion process independent of $\cal A$. 
	
	 The distribution of $H(m)$ is also related to coalescence times of semi-infinite geodesics. Indeed, let $\cc(m)$ denote the coalescence point \eqref{coal} between the semi-infinite geodesics starting at $\0$ and $(m,0)$, and with direction $(1,1)$,  and let $T(m)$ denote the second coordinate of $\cc(m)$. By self-duality of the geodesic tree \cite{Pi},  
	\begin{equation}\label{dual}  
	T(m)\stackrel{dist.}{=}H(m)\,.
	\end{equation}
	Thus, the heavy tail lower bound for $H_0$ implies that   
	$$\P\left(T(1)\geq n\right)\geq \frac{c_0}{n^{2/3}}\,.$$
	
	\subsubsection{Conjectural picture and KPZ universality}
	It is believed that the space and time fluctuations of models in the KPZ universality class can be described by variational problems involving a four parameter field $\cal A(u,x;t,y)$, where $0\leq u<t$ are time coordinates and $x,y\in\R$ are space coordinates. This field is called the \emph{space-time Airy sheet}. We address to \cite{CoQua} for a more complete description of this field and its conjectural relation with the KPZ  universality class. 
	
	In last-passage percolation models, the space-time Airy sheet would appear as the limit fluctuations of last-passage percolation times. Denote 
	$$(u,x)_n:=(\lfloor un+2^{2/3}xn^{2/3}\rfloor,\lfloor un-2^{2/3}xn^{2/3}\rfloor)\,\,\mbox{ and }\,\,(t,y)_n:=(\lfloor nt+2^{2/3}yn^{2/3}\rfloor,\lfloor nt-2^{2/3}yn^{2/3}\rfloor)\,.$$
	Then it is expected that  
	\begin{equation}\label{scalelpp1}
	\lim_{n\to\infty}\frac{L\left((u,x)_n,(t,y)_n\right)-4n(t-u)}{2^{4/3}n^{1/3}}\stackrel{dist.}{=}\cal A(u,x;t,y)-\frac{(y-x)^2}{t-u}\,.
	\end{equation}    
	For fixed times $u,t$, tightness in the space of two-dimensional continuous fields is already known \cite{CaPi1}, however no uniqueness result is available so far. 
	
	By taking $u=0$ and $t=1$, one gets a two-dimensional field $\cal A(x,y)$, called the Airy sheet. This field gives rise to a point process on the real line as follows. For each $y\in\R$ let 
	$$X(y):=\sup\arg\max_{x\in\R}\left\{\cal A(x,y)-(y-x)^2\right\}\,.$$
	The process $(X(y)\,,\,y\in\R)$ is expected to be a right-continuous pure jump process which runs through the locations of maximisers of the Airy sheet minus a drifting parabola. Notice that $X(0)$ corresponds to $X$ as defined in Theorem \ref{thm:maxscal}. The reason for taking the supremum is that there will be points $y\in\R$ such that the $\arg\max$  is not uniquely defined. However, it is not hard to see that for fixed $y$ a.s. there will be a unique maximiser \cite{Pi1} (it follows from space stationarity of the Airy sheet). This process should be similar, \textit{grosso modo}, to the Groeneboom process, which arises by taking a Brownian motion minus a drifting parabola \cite{Gro}.
	
	Related to $X(\R)=\{X(y)\,,\,y\in\R\}$ there is a stationary point process $\cal U$ defined as  
	$$\cal U(x):=\#\left(X(\R)\cap(0,x]\right)\,,$$
	which counts the number of maximisers $X(y)$ lying within the interval $(0,x]$. We conjecture that 
	\begin{equation}\label{tailconj}
	\exists\,\lim_{n\to\infty}n^{2/3}\P\left(H_0\geq n\right)=\frac{\E\cal U(1)}{2^{2/3}}\,.
	\end{equation}
	The motivation for \eqref{tailconj} is explained at the end of the proofs. In Figure \ref{simulationtree}
	we illustrate the result of crude Monte Carlo simulation\footnote{We used a random sample of size $50,000$ where the height of the tree was truncated at $n=1000$, using JULIA (through JUNO) programming language.}. It is compared with powerlaw decay taking $2.364$ as an estimate of $\E \cal U(1)$, a value suggested by those simulations.
	
		\begin{figure}[h]
			\includegraphics[height=10cm]{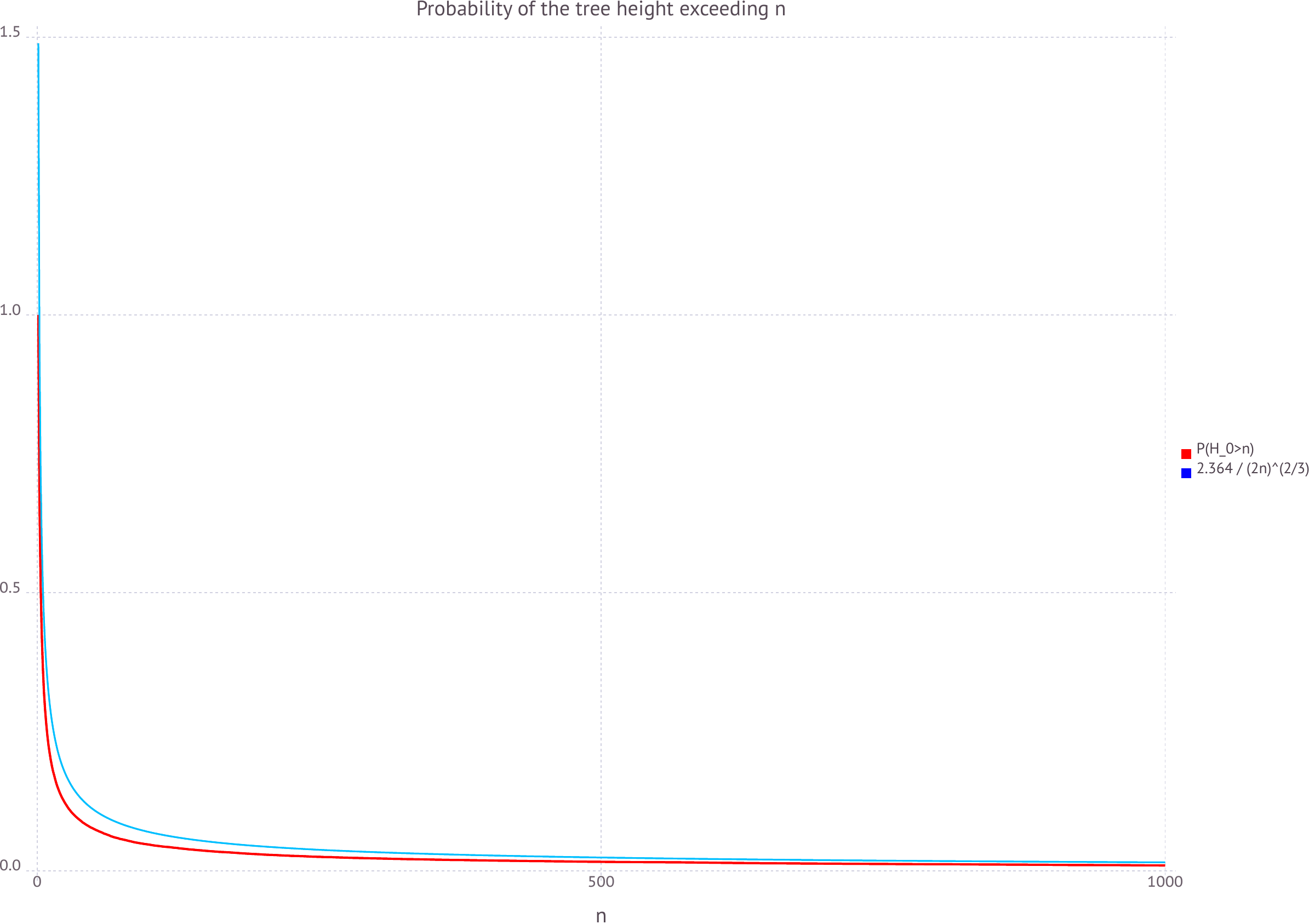}
			\caption[Estimates of $\P(H_0 \geq n)$ done by simulations.]{Estimates of $\P(H_0 \geq n)$ done by simulations.}\label{simulationtree}
		\end{figure}          
	
	The same reasoning will also yield to a conjectural description of the limiting distribution of $H(m)$ (after rescaling) in terms of $\cal U$:   
	\begin{equation}\label{scalheight1}
	\exists\,\lim_{m\to\infty}\P\left(\frac{H(m)}{2^{-1}m^{3/2}}\leq r\right)=1-\GG(r)=\P\left(\cal{U}(r^{-2/3})= 0\right)\,,
	\end{equation}
	and 
	\begin{equation}\label{scalheight2}
	\exists\,\lim_{r\to\infty}r^{2/3}\GG(r)=\E\cal U(1)\,.
	\end{equation}
	
	For the geodesic forest with an exponential weighted substrate (of parameter $1/2$) we expect to have a process $\bar{\cal U}(x):=\#\left(\bar X(\R)\cap(0,x]\right)$ where 
	$$\bar X(y):=\sup\arg\max_{x\in\R}\left\{\sqrt{2}\cal B(x)+\cal A(x,y)-(y-x)^2\right\}\,,$$
	and $\cal B(x)$ is independent a standard two-sided Brownian motion process. Thus, the conjecture will be that 
	\begin{equation}\label{tailconj1}
	\exists\,\lim_{n\to\infty}n^{2/3}\P\left(H_0> n\right)=\frac{\E\bar{\cal U}(1)}{2^{5/3}}\,,
	\end{equation} 
	\begin{equation}\label{scalheight3}
	\exists\,\lim_{m\to\infty}\P\left(\frac{H(m)}{2^{-5/2}m^{3/2}}\leq r\right)=1-\bar\GG(r)=\P\left(\bar{\cal{U}}(r^{-2/3})= 0\right)\,,
	\end{equation}
	and 
	\begin{equation}\label{scalheight4}
	\exists\,\lim_{r\to\infty}r^{2/3}\bar\GG(r)=\E\bar{\cal U}(1)\,.
	\end{equation}
	\begin{rem} We note that the limiting distribution \eqref{scalheight3} coincides with the one conjectured for coalescence times \cite{Pi}. 
	\end{rem}
	
	\section{Proofs}
	
	\subsection{The Busemman field and the location of the root}\label{Bus}
	The coalescence of (backward) semi-infinite geodesics \eqref{coal} in the direction $(1,-a)$, where $a>0$, gives rise to Busemann field defined as 
	$$B_a(\xx,\yy):=L(\cc,\yy)-L(\cc,\xx)\,,$$
	where $\cc$ is the coalescence point between $\gamma_a(\yy)$ and $\gamma_a(\xx)$. The Busemann function is an alternative construction of the equilibrium measure of last-passage percolation system $(M_t^{\nu})$ defined as $M^\nu_0=\nu$ and (recall \eqref{lppsystem})
	$$M^\nu_t(x,y]:=L_\nu(y,t)-L_\nu(x,t)\,.$$ 
	If one takes $\nu_a(k):=B_a(\0,(0,k))$ for $k> 0$ and $\nu_a(k):= B_a((0,k),\0)$ for $k\leq 0$, then $M^{\nu_a}_t\stackrel{dist.}{=}\nu_a$ for all $t\geq 0$. The i.i.d. exponential profile of parameter $1-p$ is also an equilibrium measure and this allows us get that $B_\lambda((t,k-1),(t,k))$ for $k\in\ZZ$, are i.i.d. exponential random variables of parameter  $1-\rho_a$ \cite{CaPi}. By symmetry, one can get the same result for Busemann functions in the (forward) direction $(1,a)$. By also using Burkes' property, one can describe the distribution of the Busemann function along the substrate $\phi$ as follows (see Lemma 3.3 in \cite{CaPi2}):  
	\begin{lem}\label{Busemann}
		Fix $a>0$, denote $B_a$ the Busemman field in the direction $(1,a)$ and recall the definition \eqref{rw} of the random walk $S^{a,\phi}$ . Then,
		$$\left\{B_a(\dd,\phi_z+\dd)\,:\,z\in\ZZ\right\}\stackrel{dist.}{=}\left\{S^{a,\phi}(z)\,:\,z\in\ZZ\right\}\,.$$
	\end{lem}
	
	The next ingredient in the proof is to ensure that the sequence of roots with respect to a sequence of lattice points in a direction $(1,a)$, within the rarefaction interval $(\lambda^2_-,\lambda_+^2)$, stay bounded. This is given by following lemma, which follows from Proposition 3.1 \cite{FeMaPi} (see also Lemma 3.4 \cite{CaPi2}). 
	\begin{lem}\label{rootcontrol}
		Let $a\in(\lambda^2_-,\lambda_+^2)$. Then a.s. for any sequence of lattice points $(\xx_n)_{n \geq 1}$ with asymptotic inclination $(1,a)$, there is $M>0$ such that $\Phi(\xx_n)\in\{\phi_{-M},\cdots,\phi_M\}$ for all $n\geq 1$.
	\end{lem}
	
	\noindent\paragraph{\bf Proof of Theorem \ref{rootlaw}} By \eqref{rootmax},
	$$\Phi(\xx_n)=\argmax _{\zz\in\cal{C}_{\xx_n}(\phi)} L(\zz+\dd,\xx_n)=\argmax _{\zz\in\cal{C}_{\xx_n}(\phi)}\left( L(\zz+\dd,\xx_n)-L(\dd,\xx_n)\right)\,.$$
	By Lemma \ref{rootcontrol},
	$$\lim_{n\to\infty}\argmax _{\zz\in\cal{C}_{\xx_n}(\phi)}\left( L(\zz+\dd,\xx_n)-L(\dd,\xx_n)\right)\stackrel{dist.}{=}\argmax _{\zz\in\cal C(\phi)}B_a(\zz+\dd,\dd)\,.$$
	Together with Lemma \ref{Busemann}, this implies Theorem \ref{rootlaw}.
	
	\subsection{Computable models}

\noindent\paragraph{\bf Proof of Proposition \ref{ZrelN}} Define the process $\{W_n\}_{n \leq 0}$ by
\begin{equation*}
W_{-n}:= \max_{-k \leq -n} ( S_k - S_n) = -S_n - \min_{-k \leq -n} (-S_k)   \qquad \forall n \geq 0,
\end{equation*}
and the random variable $\tau:= \min \{ n \geq 0 : W_{-n} = 0 \}$.  We prove now that $(Z,M)=(\tau, W_0)$. \\

For $n=0, x \geq 0$ we have r
$$\{ \tau=0, W_0= x \}  = \{ S_0=0,  S_0= \max_{ n \geq 0 } S_n,  x = 0 \}  = \{ Z= 0 , M= x  \} .$$
In the case of $n \geq 1, x \geq 0$, it holds
\begin{eqnarray*}
	\{ \tau=n, W_0 >x \} &=& \{ X_n > 0, ..., X_n + \cdots +X_1 > 0, X_n + \cdots  +X_1 > x, S_n = \max_{ k \geq n } S_k \}  \\
	&=& \{ S_n - X_n <  S_n, ...,  S_n - (X_n+\cdots+X_2) < S_n, S_n >x, S_n = \max_{ k \geq n } S_k  \} \\
	&=& \{ S_{n-1} < S_n, ...,  S_{1} < S_n, S_n >x, S_n = \max_{ k \geq n } S_k \} \\
	&=&\{ S_n > x, S_n  = \max_{ k \geq 0 } S_k \} =\{ Z =n , M >x \}.
\end{eqnarray*}
By construction the process $\{W_n\}_{n \leq 0}$ is stationary. Besides, it satisfies \textit{Lindley property}:
\begin{equation}\label{Lindley}
W_{-n}= (W_{-n-1} + X_{n+1})^*   \qquad \forall n \geq 0,
\end{equation}
where $X_{n+1}$ and $W_{-n-1}$ are independent random variables and $x^*= \max \{0,x\}$, for $x \in \mathbb{R}$. A set of random variables which satisfy recursion \eqref{Lindley} has an interpretation as sequential waiting times in queueing theory \cite{Phi}. By definition of $\tau$, for all $x \geq 0$ we have that 
$$\P(Z=0,M=x)=\P (\tau =0, W_0 \in [x, x+dx) ) = 1_{ \{   x=0 \} } \P (W_0=0), $$ 
while for $n \geq 1$, $x > 0$ we have
\begin{eqnarray*}
	\P(Z=n, M \in [x, x+dx)) &=& \P (\tau = n, W_0 \in [x, x+dx) ) \\
	 &=& \P (W_{-n}=0, W_{-n+1}>0,..., W_{0} >0, W_0 \in [x, x+dx)) \\
	& = & \P (W_{-n+1}>0, \dots,W_{0} >0, W_0 \in [x, x+dx) \Big| W_{-n} = 0 ) \P ( W_{-n} = 0 ) \\
	&=& \P ( X_{n} >0,\dots, X_{n} +\cdots + X_{1} >0, X_{n} +\cdots + X_{1}  \in [x, x+dx) )  \\
	&   & \cdot \, \P ( W_{-n} = 0 ),  \\
	& =& \P (N>n, S_n  \in [x, x+dx) ) \P ( W_{-n} = 0 ), 
\end{eqnarray*}
where we used \eqref{Lindley}. Since $\{W_{n}\}_{n \leq 0}$ is stationary, $W_{-n}$ is equal in law to $M$ and we are done. 
\newline
\noindent\paragraph{\bf Proof of Proposition \ref{jointexpl} } Define the function
$$ \phi_{Z^\pm_a, M^\pm_a }(s,u) := \sum_{n=1}^{\infty} \int_{0}^{\infty} s^n \, e^{ux} f_{Z^\pm_a, M^\pm_a}(n,x) dx. $$
By Proposition \ref{ZrelN}, we can express $\phi_{Z^\pm_a, M^\pm_a }(s,u)$ as
\begin{eqnarray}\label{sumtrans}
\phi_{Z^\pm_a, M^\pm_a }(s,u)  &= &  \sum_{n=1}^{\infty}  \int_{0}^{\infty} s^n e^{ux}   \sum_{j>n}f_{N^\pm_a, S^{a,\pm}_n}(j,x) \P(M^\pm_a=0) dx \nonumber\\
&= & \P(M^\pm_a=0)  \sum_{n=1}^{\infty}  s^n   \P(N^\pm_a >n)  \int_{0}^{\infty}  e^{ux} f_{ (S^{a,\pm}_n =x | N^\pm>n) } dx \nonumber \\
&= & \P(M^\pm_a=0)   \sum_{n=1}^{\infty}  s^n   \P(N^\pm_a >n)  \E( e^{u S^{a,\pm}_n } | N^\pm_a>n ) .
\end{eqnarray}
where all the variables with signs $^\pm$ are the corresponding ones to each one-sided random walk. By \cite{Ig} (p. 744), for all $u \in [0,c]$, and  $0< s < 1$  we have that
$$ \sum_{n=1}^{\infty}  s^n   \P(N_a^\pm >n)  \E( e^{u S^{a,\pm}_n } | N_a^\pm>n )  = (1-s) \exp \Big\{  \sum_{n=1}^{\infty}  \frac{s^n}{n} [ \P(S^{a,\pm}_n>0) + \E(e^{-u(S^{a,\pm}_n)^*})  ]  \Big\} ,$$
then by plug it in (\ref{sumtrans}) and expressing $1-s = \exp \{ - \sum_{n=1}^{\infty} \frac{s^n}{n} \}$, we obtain
\begin{equation}\label{phiZM}
\phi_{Z^\pm_a, M^\pm_a }(s,u) = \P(M^\pm_a=0)  \exp \Big\{  \sum_{n=1}^{\infty}  \frac{s^n}{n} [ \E(e^{-u(S^{a,\pm}_n)^*}) - \P(S^{a,\pm}_n \leq 0)  ]  \Big\}=\phi^\pm_a(s,u)\,.
\end{equation}
Recall the definition of the variables 
$ Z^+_a = \argmax_{ k\geq 0 } S_k^{a,+}$ , $Z^-_a = -\argmax_{k \leq 0}  S_k^{a,-}$.
Then, the joint transform of $(Z_a,M_a)$ can be written as
\begin{eqnarray}\label{joint2side}
\E (s^{Z_a} e^{uM_a} ) &=& \P( M_a^+=0) \P (M_a^-=0) + \sum_{n=1}^\infty \int_{0}^{\infty} s^n e^{ux} f_{Z^+_a,M^+_a}(n,x) \P( M^-_a<x) dx  \\
&+& \sum_{n=-\infty}^{-1} \int_{0}^{\infty} s^n e^{ux} f_{Z^-_a,M^-_a}(-n,x)  \P(M^+_a<x) dx. \nonumber 
\end{eqnarray}
As a consequence of the exponential right tail assumption \eqref{maxdist} we have 
$$\P (M^\pm_a \leq x)= 1 - \Big(1- \frac{\gamma^\pm}{\delta^\pm} \Big) e^{- \gamma^\pm \, x} \qquad \forall x > 0,$$
thus by substituing it in \eqref{joint2side} and factorising the exponents, 
\begin{eqnarray*}
\E(s^{Z_a} \, e^{uM_a})  &=& \Big( \frac{\gamma^+}{\delta^+} \Big) \Big( \frac{\gamma^-}{\delta^-} \Big) +\phi_{Z^+_a, M^+_a}(s,u) 
- \Big( 1 - \frac{\gamma^+}{\delta^+} \Big) \phi_{Z^+_a, M^+_a}(s,u - \gamma^+) \\
&+&  \phi_{Z^-_a, M^-_a}(s,u) - \Big( 1 - \frac{\gamma^-}{\delta^-} \Big)  \phi_{Z^-_a, M^-_a}(s,u- \gamma^-), \nonumber
\end{eqnarray*}
and the result follows since we already calculated $\phi_{Z^\pm_a, M^\pm_a}=\phi^\pm_a$  in \eqref{phiZM}.

	\subsection{Scaling the height function and the root counting process}
	The key in the proof of the theorems for the scaling scenario is the introduction of a point process that counts the number of roots whose geodesic tree has height bigger or equal to $n$. Define the \emph{root counting process} at ``time'' $n\geq 1$ as
	$$\cal{Z}_n(m):=\sum_{z\in (0,m]} \zeta_n(z)\,,\,\mbox{ for }m\geq 1\,,$$ 
	where
	$$\zeta_n(z):=\left\{\begin{array}{ll}1 & \mbox{ if $(z,-z)=\Phi(x_1,x_2)$ and $x_1+x_2=n+1$},\\ 
	0 & \mbox{otherwise}\,.\end{array}\right.$$
	In words, $\zeta_n(z)=1$ if $(z,-z)$ is a root of a tree that intersects the line $x_1+x_2=n+1$. The process $\cal Z_n$ counts the number of such roots in the interval $(0,m]$. Notice that, by translation invariance of the last-passage percolation model, the counting process $\cal Z_n$ is also stationary. It is also not hard to see that $H(m)$, the maximum height on $(0,m]$, and this counting process are related by 
	$$\{H(m)< n\}=\{\cal Z_n(m)=0\}\,,$$
	and so 
	\begin{equation}\label{counting2}
	\P\left(H(m)< n\right)=\P\left(\cal Z_n(m)=0\right)\,.
	\end{equation}
	In particular,
	\begin{equation}\label{counting3}
	p_n:=\P\left(\zeta_n(0)=1\right)=\P\left(H_0\geq n\right)\,.
	\end{equation}
	
	We also note that, by definition, $z$ is a root if and only if it is a point to line maximiser \eqref{rootmax}: $\zeta_n(z)=1$ if and only if there exists $\xx=(x(1),x(2))$ such that $x(1)+x(2)=n+1$ and 
	\begin{equation}\label{maximiser}
	z=\arg\max_{y:(y,-y)\leq \xx} L\left((y+1,-y+1),\xx\right)\,.
	\end{equation}
	Thus, $\cal Z_n$ can also be seen as point process that counts the number of maximisers at ``time'' $n$. 
	\newline  
	
	\noindent\paragraph{\bf Proof of Theorem \ref{thm:tail}} 
	Let $Z(n)$ denote location in $\ZZ$ of the root of $(n,n)$: $\Phi(n,n)=(Z(n),-Z(n))$. By \eqref{scalelpp}, we have that 
	\begin{equation}\label{scalarg}
	\lim_{n\to\infty}\frac{Z(n)}{2^{2/3}n^{2/3}}\stackrel{dist.}{=}X:=\arg\max_{x\in\R}\left\{\cal A(x)-x^2\right\}\,.
	\end{equation}
	(This is analog to Theorem 1.6 in \cite{Jo}.) In particular, there exist $\epsilon,c>0$ such that 
	$$\P\left(|Z(n)|< cn^{2/3}\right)>\epsilon\,,$$
	for all $n\geq 1$. If $|Z(n)|< m$ then there will at least one root in $(-m,m)$. By stationarity of $\cal Z_n$, we then have that
	$$\P\left(|Z(n)|< m\right)\leq \P\left(\cal Z_n(2m-1)\geq 1\right)\leq 2m p_n\,,$$
	(in the right-hand side inequality we use the union bound) and hence, 
	$$0<\epsilon<\P\left(|Z(n)|\leq cn^{2/3}\right)\leq 2cn^{2/3}p_n\,,$$
	which implies that
	$$\P\left(H_0\geq n\right)=p_n\geq \frac{\epsilon}{2cn^{2/3}}\,,$$
	and finishes the proof of Theorem \ref{thm:tail}.
	\newline
	
	\noindent\paragraph{\bf Proof of Theorem \ref{thm:maxscal}} We use again \eqref{counting2}: 
	$$\P\left(|Z(n)|< m\right)\leq \P\left(\cal Z_n(2m-1)\geq 1\right)\leq \P\left(H(2m)\geq n\right)\,,$$
	and thus,
	$$\P\left(\frac{H(m)}{2^{-1}m^{3/2}}\geq r\right)\geq \P\left( \frac{Z(n)}{2^{2/3}n^{2/3}}\in\left(\frac{-1}{2r^{2/3}},\frac{1}{2r^{2/3}}\right)\right)\,,$$
	where $n=\lfloor r2^{-1}(2m)^{3/2}\rfloor$, which shows that 
	$$\GG(r)\geq \P\left( X \in\left(\frac{-1}{2r^{2/3}},\frac{1}{2r^{2/3}}\right]\right)\,.$$
	\newline
	
	\begin{rem} The reason for conjectures \eqref{tailconj}, \eqref{scalheight1} and \eqref{scalheight2}, lyes in the expected limiting behaviour of the point process $\cal Z_n$. Indeed, by seeing roots as maximisers \eqref{maximiser}, if \eqref{scalelpp1} is true then we must have that 
		$$\lim_{n\to\infty}X_n(y)\stackrel{dist.}{=}X(y)\,\,\mbox{ (as process in $y\in\R$) }\,,$$
		where 
		$$X_n(y):=\argmax_{x}\left\{\frac{L\left((0,x)_n,(1,y)_n\right)-4n}{2^{4/3}n^{1/3}}\right\}\,.$$
		Thus, under assumption \eqref{scalelpp1}, we have
		$$\lim_{n\to\infty}\cal Z_n(2^{2/3}xn^{2/3})\stackrel{dist.}{=}\cal U(x)\,.$$
		To get \eqref{tailconj} one also needs to use \eqref{counting3}. For the scaling behaviour of the weighted substrate model we address the reader to \cite{Pi}.
	\end{rem}

\end{document}